\documentclass[10pt,a4paper]{article}

\usepackage[latin1]{inputenc}
\usepackage[english]{babel}
\usepackage{enumerate}
\usepackage[usenames, dvipsnames]{xcolor}
\usepackage{amssymb}
\usepackage{amsfonts}
\usepackage{amsmath}
\usepackage{amsthm}
\usepackage{epsfig} 
\usepackage{graphics}
\usepackage{dsfont}

\usepackage{verbatim}

\usepackage{latexsym}

\newtheorem{theorem}{Theorem}[section]
\newtheorem{lemma}[theorem]{Lemma}

\newtheorem{corollary}[theorem]{Corollary}

\theoremstyle{definition}
\newtheorem{remark}[theorem]{Remark}

\newtheorem{example}[theorem]{Example}

\theoremstyle{definition}

\allowdisplaybreaks[4]

\newcommand{\RR}{\mathbb{R}}

\newcommand{\EE}{\mathbb{E}}

\newcommand{\Fskript}{\mathcal{F}}
\newcommand{\PP}{\mathbb{P}}
\newcommand{\cB}{\mathcal{B}}
\newcommand{\cC}{\mathcal{C}}

\newcommand{\cA}{\mathcal{A}}

\newcommand{\R}{\RR}
\newcommand{\Prob}{\PP}
\newcommand{\eqd}{\stackrel{d}{=}}
\newcommand{\dtv}{d_{TV}}
\newcommand{\E}{\EE}
\newcommand{\One}{\mathds 1}
\newcommand{\Bskript}{\cB}
\numberwithin{equation}{section}

\begin{document}
\title{Markovian Maximal Coupling of Markov Processes}
\author{Bj\"orn B\"ottcher \thanks{Technische Universit\"at Dresden, 
Institut f\"ur Mathematische Stochastik, D-01062 Dresden, Germany}}
\date{November 3, 2016\footnote{Date of submission to journal.}}
\maketitle

\begin{abstract}

Markovian maximal couplings of Markov processes are characterized by an equality of total variation and a distance of Wasserstein type. If a Markovian maximal coupling is a Feller process, the generator can be calculated, e.g. for reflection coupled Brownian motion. Apart from processes with continuous paths also jump processes are treated for the first time. For subordinated Brownian motion a Markovian maximal coupling is constructed by subordinating reflection coupled Brownian motion. This coupling is the unique Markovian maximal coupling and its generator is determined by state-space dependent mirror coupling of the corresponding L\'evy measures.

\end{abstract}



\section{Introduction}
Coupling methods are a powerful probabilistic tool, for an (extensive) overview see e.g. \cite{Lind1992, Thor2000}.
For two stochastic processes on a common state space any joint distribution is a coupling of these processes and the first time their paths meet is called coupling time.
Roughly speaking, a maximal coupling minimizes the coupling time of the coupled processes and it is used e.g. to study ergodicity, convergence rates or the spectral gap, cf. \cite{BurdKend2000},\cite[Remark 2.4]{Kuwa2009}. The first construction of a maximal coupling was given by Griffeath \cite{Grif1975} (see also Goldstein \cite{Gold1979}) for Markov chains and later extended to continuous time and continuous state-space by Sverchkov and Smirnov \cite{SverSmir1990}.

Constructions for diffusion processes (on $\R^d$ and on manifolds) have been the focus of recent research, e.g. Kuwada, Hsu, Sturm, Banerjee and Kendall \cite{Kuwa2009,HsuStur2013,BaneKend2014,BaneKend2015}. 

Besides maximality further properties of a coupling are of interest: A coupling of Markov processes is called {Markovian} if it is a Markov process and -- slightly more general -- it is called {co-adapted} if the coupled components are Markov processes with respect to the joint filtration. Here we use the terminology as in \cite{BurdKend2000,Kend2010,Conn2013}. Instead of co-adapted also the term Markovian \cite{BaneKend2014, HsuStur2013,Kuwa2009} has been used. Closely related concepts are {faithful} \cite{Rose1997} and {immersed} \cite{Kend2015, BaneKend2015} couplings.

Existence and construction of Markovian or co-adapted maximal couplings has been discussed in some particular cases. Hsu and Sturm \cite{HsuStur2013} proved that the reflection coupling is the unique co-adapted maximal coupling of Brownian motions in $\R^d$. Kuwada \cite{Kuwa2009} gave an example of a Markov chain on a discrete state-space for which no co-adapted maximal coupling exists. He also showed that any co-adapted maximal coupling of Brownian motions on a Riemannian manifold has to be a reflection coupling with an appropriate reflection structure. Connor \cite[Lemma 3.12]{Conn2007} showed that for Brownian motion with non Dirac initial distribution, in general, no co-adapted maximal coupling exists. Banerjee and Kendall \cite{BaneKend2014} discussed the existence of a co-adapted maximal coupling for diffusions (under some regularity conditions). They presented necessary and sufficient conditions on the drift.

Chen and Li studied generators of Markovian couplings, e.g. \cite{ChenLi1989,Chen1986,Chen1994a}. In particular, optimality (which is in some sense a generalization of maximality) of the couplings with respect to certain Wasserstein distances was discussed.

In general, \textit{maximal couplings are, in most cases, unintuitive non-Markovian affairs, and extremely difficult to work with} \cite[p. 1118]{Conn2013}. However, if the maximal couplings are Markovian they are easy to work with. But so far only maximal Markovian couplings are known for processes with continuous paths or in discrete time. Extending the ideas of \cite{HsuStur2003, HsuStur2013} we construct and show the uniqueness of Markovian maximal couplings for particular jump processes, namely subordinated Brownian motions. Furthermore, using the theory of Feller processes we are able to calculate the generator of the Markovian maximal coupling.

In the next section the basics of (maximal) coupling are introduced. Section \ref{sec:charmmc} provides the characterization of Markovian maximal couplings. Thereafter subordinated Brownian motions are coupled and finally the generators of Markovian couplings are discussed for Feller processes. In an appendix (Section \ref{appendix}) some non-standard results for L\'evy and Feller processes are summarized.

The results of the paper can also be summed up in the following way: In Theorem \ref{thm:mmc-characterization} (Equation \eqref{dtv-markovian-maximal-eq}) a Markovian maximal coupling is characterized using a Wasserstein type distance.  This equation together with the theory of optimal transport, e.g. \cite{Vill2009}, can be used to prove uniqueness of the maximal coupling, in particular for subordinated Brownian motion (Theorem \ref{thm:mmc-subordinated-bm-uniqueness}). If a Markovian coupling is a Feller process (Theorem \ref{thm:rcbm-feller} and Corollary \ref{thm:srcbm-feller}) one can calculate the generators of the couplings considered: for reflected Brownian motion \eqref{eq:rcbm-generator} and for subordinated Brownian motion \eqref{eq:mmc-feller-generator}.

\section{Basics}

We start with maximal coupling of random variables (see e.g. \cite[Section 1.4]{Thor2000}). Let $X,Y$ be random variables on a common probability space $(\Omega,\cA,\PP)$ with distributions $\PP_X$ and $\PP_Y$, respectively. Any pair $(\widetilde{X},\widetilde{Y})$ is called a \textbf{coupling} of $X,Y$ if $\widetilde{X} \eqd X$ and $\widetilde{Y} \eqd Y$. This will be denoted by $(\widetilde{X},\widetilde{Y})\in \cC(X;Y)$. Here $\eqd$ denotes the equality of the distributions. For $(\widetilde{X}^M,\widetilde{Y}^M) \in \cC(X;Y)$ the set $\{\widetilde{X}^M=\widetilde{Y}^M\}$ is called \textbf{coupling event} and the coupling is called \textbf{maximal} if it maximizes the probability of the coupling event, i.e.,
\begin{equation} \label{max-coupling}
\PP(\widetilde{X}^M=\widetilde{Y}^M) = \sup_{(\widetilde{X},\widetilde{Y})\in \cC(X;Y)} \PP(\widetilde{X}= \widetilde{Y}).
\end{equation}
A maximal coupling always exists (cf. for example \cite[Section 3.1]{Conn2007}), the construction is sometimes called Wasserstein (Vasershtein) coupling. If $X$ and $Y$ have densities $f$ and $g$ with respect to some measure $\lambda$, respectively, then 
\begin{equation}
\PP(\widetilde{X}^M=\widetilde{Y}^M) = \int f\land g\, d\lambda,
\end{equation}
where $a\land b:= \min(a,b)$.

Note that a maximal coupling only fixes the probability of the set on which the random variables coincide, but when they do not coincide the joint distribution is not determined. Thus, in general, a maximal coupling is not unique. As a side remark, note that any coupling can be constructed using a copula but there is no fixed maximal coupling copula \cite{MaiSche2014}. 
We define the total variation distance of $X$ and $Y$ by
\begin{equation}
\dtv(X,Y):= \sup_{A\in\cA} |\PP_X(A)-\PP_Y(A)|.
\end{equation}
In terms of probability distances, $\dtv$ is the minimal distance of the compound distance $d_I(X,Y):=\Prob( X \neq Y)$, i.e. 
\begin{equation} \label{tv_inf}
\dtv(X,Y) = \inf_{(\widetilde{X},\widetilde{Y})\in \cC(X;Y)} \PP(\widetilde{X}\neq \widetilde{Y})
\end{equation}
and thus
\begin{equation}
\label{tv_eq}
\dtv(X,Y) = \PP(\widetilde{X}^M \neq \widetilde{Y}^M)
\end{equation}
for any maximal coupling $(\widetilde{X}^M,\widetilde{Y}^M)$ of $X$ and $Y$.
Moreover, if $X$ has a density $f$ and $Y$ has a density $g$ with respect to a measure $\lambda$ then
\begin{equation} \label{dtv-density}
\dtv(X,Y)  = \int (f-g)^+\, d\lambda= \frac{1}{2} \int |f-g|\, d\lambda= 1-\int f \land g\, d\lambda .
\end{equation}

For stochastic processes $(X_t)_{t\geq 0}$ and $(Y_t)_{t\geq 0}$ on a common probability space we define by
\begin{equation}
T(X_.,Y_.) := \inf\{t\geq 0\ :\ X_t=Y_t\}
\end{equation} 
their \textbf{coupling time}. If both processes are strong Markov processes with the same transition probabilities and if for fixed initial distributions there exists a coupling $(\widetilde{X}_.,\widetilde{Y}_.)\in \cC(X_.;Y_.)$ with
\begin{equation}
\Prob(T(\widetilde{X}_.,\widetilde{Y}_.)<\infty) = 1,
\end{equation}
or equivalent
\begin{equation}
\lim_{t\to \infty} \Prob(T(\widetilde{X}_.,\widetilde{Y}_.)> t) = 0,
\end{equation}
then this coupling is called \textbf{successful}. The process $(X_t)_{t\geq 0}$ is said to possess the \textbf{coupling property}, if a successful coupling exists for all initial distributions. In this setting, due to the strong Markov property, the coupling can be chosen such that the paths of the two processes coincide after they meet and for this coupling the well known coupling time inequality \cite[p. 35, (2.1)]{Thor2000} (also known as coupling inequality \cite[section I.2]{Lind1992} or Aldous inequality \cite{BaneKend2014})
\begin{equation} \label{eq:couplingtime}
\PP(T(\widetilde{X}_.,\widetilde{Y}_.)>t) = \Prob(\widetilde{X}_t\neq \widetilde{Y}_t) \geq \dtv(X_t,Y_t)
\end{equation}
holds for all $t\geq 0$. Moreover, for processes one calls the coupling $(\widetilde{X}^M_.,\widetilde{Y}^M_.)\in \cC(X_.;Y_.)$ \textbf{maximal coupling} if for each fixed time the coupling is maximal, i.e. by \eqref{tv_eq}
\begin{equation} \label{eq:def:maxcoupling-processes}
\dtv(X_t,Y_t) = \PP(\widetilde{X}^M_t \neq \widetilde{Y}^M_t) \text { for all $t$.}
\end{equation} 
This coupling is the \textbf{unique maximal coupling} if any maximal coupling  $(\widehat{X}^M_.,\widehat{Y}^M_.)\in \cC(X_.;Y_.)$ is distributed as $(\widetilde{X}^M_.,\widetilde{Y}^M_.)$.

Whenever a maximal coupling exists $(X_t)_{t\geq 0}$ possesses the coupling property if and only if
\begin{equation}
\dtv(X_t,Y_t) \xrightarrow{t\to \infty} 0 \text{ for all initial distributions.}
\end{equation}

We have two different definitions of maximal coupling, we use \eqref{max-coupling} for random variables and \eqref{eq:def:maxcoupling-processes} for processes. Maximal coupling as in \eqref{max-coupling} would also make sense for path valued random variables, but in this setting for processes with initial distributions with disjoint supports any coupling would be a maximal coupling since their total variation distance is 1. Nevertheless \eqref{max-coupling} can be related to \eqref{eq:def:maxcoupling-processes} by considering the shifted processes \cite[Lemma 2.3]{Kuwa2009}.

Let $(Z_t)_{t\geq 0}$ be a Markov process and let $X$ and $Y$ be copies of $Z$ with possibly different initial distributions. Note that we abuse notation here, since we omit the corresponding family of distributions, the filtration and the shift operators. We denote by $(Z_t^{x,r})_{t\geq 0}$ the process started in $x$ at time $r$, for a time-homogeneous process this simplifies to $(Z_t^{x})_{t\geq 0}$.
 A coupling $(\widetilde{X}_.,\widetilde{Y}_.)\in \cC(X_.;Y_.)$ is called \textbf{co-adapted} if $(\widetilde{X}_t)_{t\geq 0}$ and $(\widetilde{Y}_t)_{t\geq 0}$ are Markov processes with respect to the natural filtration of the joint process, i.e., $(\Fskript_t)_{t\geq 0}$ with $\Fskript_t := \sigma\{(\widetilde{X}_s,\widetilde{Y}_s)\,:\, s\leq t\})$.
It is called \textbf{Markovian} if $((\widetilde{X}_{t},\widetilde{Y}_{t}))_{t\geq 0}$ is a Markov process with respect to $(\Fskript_t)_{t\geq 0}$. Thus any Markovian coupling is also co-adapted. Note that a Markovian coupling of time-homogeneous Markov processes can be time-inhomogeneous, cf. Example \ref{ex:coupling-time-inhomo-not-strong-mp}. But a unique Markovian maximal coupling of time-homogeneous processes is time-homogeneous, cf. Theorem \ref{thm:max-coupling-time-homo}.

Hsu and Sturm \cite{HsuStur2013} showed that for Brownian motion the mirror coupling is the only maximal coupling which is also co-adapted.  
In fact they used (see also \cite{HsuStur2003}) the term mirror coupling for two properties (which are equivalent in their setting):
\begin{enumerate}
\item On the one hand one can couple the paths of spatial homogeneous and symmetric Markov processes $(X_t)_{t\geq 0}, (Y_t)_{t\geq 0}$ with continuous paths starting from $x$ and $y$, respectively, by taking, until they meet,  $Y_t$ to be $X_t$ reflected with respect to the plane which contains the midpoint of $x$ and $y$ and to which $x-y$ is normal. After hitting this plane one sets $Y_t:=X_t$.  This will be called \textbf{reflection coupling} (as in \cite{BoetSchiWang2011a}, it dates back to \cite{Lind1982,LindRoge1986}). Formally for $x\neq y$ the reflection plane is \begin{equation}
H_{x,y}:= \{z\,:\, |x-z|=|y-z|\},
\end{equation} the hitting time of the plane is $\tau_{H_{x,y}} := \inf\{t>0\,:\, X_t \in H_{x,y}\}$ and the reflection on $H_{x,y}$ is (cf. \cite[p. 1204]{BoetSchiWang2011a})
\begin{equation} \label{eqdef:reflection}
R_{x,y} z := z - 2 \langle z-\frac{x+y}{2}, x-y \rangle \frac{x-y}{|x-y|^2} \text{ for } x\neq y.
\end{equation}
(Note that $y\mapsto R_{x,y}z$ has (for $d>1$) no continuous extension to $y=x$, since on the one hand $y_1=x_1, y_2\to x_2$ yields $R_{x,x}z = (z_1,-z_2+2x_2)$ and on the other hand $y_2=x_2, y_1\to x_1$ yields $R_{x,x}z = (-z_1+2x_1,z_2)$. We set $\tau_{H_{x,x}} := 0$ and $R_{x,x} z = z$, the latter coincides in \eqref{eqdef:reflection} with the convention $\frac{0}{0} :=  0.$) Then the reflection coupling is
\begin{equation}
(X_t,(R_{x,y}^\tau X)_t) \text{ where } (R_{x,y}^\tau X)_t := \begin{cases}
R_{x,y} X_t &,\text{ if } t< \tau_{H_{x,y}},\\ X_t &, \text{ if } t\geq \tau_{H_{x,y}}.
\end{cases}
\end{equation}

\item On the other hand, let $H_{x,y}^0:=\{z\,:\,\langle z,x-y \rangle = 0\},$ thus $H_{x,y}^0$ contains the origin and it is parallel to $H_{x,y}.$ Let $f$ be a probability density on $\R^d$ which is symmetric with respect to $H_{x,y}^0$, i.e.,
\begin{equation}
f(z) = f\left(z-2\langle z,x-y\rangle \frac{x-y}{|x-y|^2}\right).
\end{equation}
Let $X$ and $Y$ be random variables with densities $f(\cdot -x)$ and $f(\cdot - y)$ with respect to the Lebesgue measure $\lambda^d$ for some $x$ and $y$, respectively. Then for $B\in\Bskript(\R^{2d})$, i.e. a Borel set $B\subset\R^{2d}$,
\begin{equation} \label{mirror-coupling-distribution}
\Prob((\widetilde{X},\widetilde{Y}) \in B) = \int \One_B(z,z) (f(z-x)\land f(z-y))\,\lambda^d(dz) + \int \One_B(z,R_{x,y}z) (f(z-x)- f(z-y))^+\, \lambda^d(dz)
\end{equation}
defines a coupling $(\widetilde{X},\widetilde{Y}) \in \cC(X;Y)$ such that
\begin{equation}
\label{mirror-coupling-ident}
\Prob(\widetilde{X}=\widetilde{Y})= \int f(z-x ) \land f(z-y) \, \lambda^d(dz)
\end{equation}
and if $\widetilde{X}\neq\widetilde{Y}$ they are mirror symmetric with respect to $H_{x,y}$, i.e.
\begin{equation} \label{mirror-coupling-mirror}
\Prob(\widetilde{Y} = R_{x,y}\widetilde{X}\ |\ \widetilde{X}\neq\widetilde{Y}) = 1.
\end{equation}
Note that \eqref{mirror-coupling-ident} and \eqref{mirror-coupling-mirror} together determine \eqref{mirror-coupling-distribution}.

We call the coupling given by \eqref{mirror-coupling-distribution} \textbf{mirror coupling} and by \eqref{mirror-coupling-ident} it follows that it is a maximal coupling. A coupling of processes is a mirror coupling if the laws for each fixed time are mirror coupled. On the one hand \eqref{mirror-coupling-distribution} can be generalized to measures without densities (by the Hahn decomposition theorem). On the other hand, if we further assume that $f(x)=p(|x|)$ and $p:\RR \to [0,\infty)$ is monotone on $[0,\infty),$ e.g. $f$ is rotationally invariant and unimodal, then
\begin{equation} \label{mirror-coupling-density}
1 - \int f(z-x ) \land f(z-y) \, \lambda^d(dz) = 2 \int_0^\frac{|x-y|}{2} \int_{\R^{d-1}} p\left(\sqrt{ z_1^2 + z_2^2 + \ldots + z_d^2}\right)\,\lambda^{d-1}(dz_2\ldots z_d)\, \lambda^1(dz_1).
\end{equation}
For $d=1$ this simplifies to
\begin{equation}
 1 - \int f(z-x ) \land f(z-y) \, \lambda^1(dz) = 2 \int_0^\frac{|x-y|}{2} p\left(z\right)\, \lambda^1(dz).
\end{equation}

\end{enumerate}
Note that reflection coupling is a statement about the path, while mirror coupling is a statement about single (fixed time) distributions. The former does not have a clear extension to jump processes, since jump processes usually do not hit the reflection plane. As we will see, the mirror coupling can be used for jump processes. Let us first follow \cite{HsuStur2013} and look at Brownian motion. 

Let $(\widetilde B^x_t,\widetilde B^y_t)_{t\geq 0}$ be the reflection coupling of $\R^d$-valued Brownian motions started in $x$ and $y$, respectively. Then (e.g. \cite[section 5]{ChenLi1989} or using the distribution of the hitting time of the half space, whose Laplace transform is known cf. \cite[Theorem 5.13]{SchiPart2014})
\begin{equation}\label{bm-couplingtime}
\Prob(T(\widetilde B_.^x,\widetilde B_.^y) > s) = 2 \int_0^\frac{|x-y|}{2} g_{s}(z)\, \lambda^1(dz)
\end{equation}
for all $s> 0$, where $g_{s}$ denotes the density of the normal distribution (on $\R$) with mean 0 and variance $s$. By \eqref{mirror-coupling-density} and \eqref{dtv-density} (see \eqref{eq:tv-subordinated} for the general case) it follows that the right hand side is the total variation distance of $B_s^x$ and $B_s^y$ and thus reflection coupling of Brownian motion is a maximal coupling. Calculating the distribution of $(\widetilde B^x_t,\widetilde B^y_t)$ yields that they are mirror coupled for each fixed $t$. Hsu and Sturm \cite{HsuStur2013} use a particular Wasserstein distance to show that the reflection coupling is the unique co-adapted maximal coupling. We will come back to this in section \ref{sec:maxsBM}.

\section{Characterization of Markovian maximal couplings} \label{sec:charmmc}

Eventually we will be interested in time-homogeneous processes, but first we give a characterization of Markovian maximal couplings in the general setting.

\begin{theorem} \label{thm:mmc-characterization}
Let $(X_t)_{t\geq 0}$ be a Markov process on $\R^d$ and let $x,y\in\R^d$, $s,t>0$, $r\geq 0$.
\begin{enumerate}
\item For any coupling $(\widetilde X^{x,r}_{t},\widetilde X^{y,r}_t)$ of $X_t^{x,r}$ and $X_t^{y,r}$ 
\begin{equation} \label{dtv-immersed-lower-bound}
\dtv(X_{t+s}^{x,r},X_{t+s}^{y,r}) \leq \int \dtv(X_s^{{x_1,t+r}}, X_s^{{y_1,t+r}})\, \Prob\left((\widetilde X^{x,r}_{t},\widetilde X^{y,r}_t)\in d(x_1, y_1)\right).
\end{equation}
\item Suppose that $(\widetilde X_{t},\widetilde X_t)_{t\geq 0}$ is a Markovian coupling of $(X_t)_{t\geq 0}$, i.e. it is a Markov process and for any starting positions $x, y\in\R^d$, $r\geq 0$ the process $((\widetilde X^{x,r}_{t},\widetilde X^{y,r}_t))_{t\geq 0}$ is a coupling of $(X_t^{x,r})_{t\geq 0}$ and $(X_t^{y,r})_{t\geq 0}$, then
\begin{equation}  \label{dtv-markovian-lower-bound}
 \int \dtv(X_s^{x_1,t+r}, X_s^{y_1,t+r})\, \Prob\left((\widetilde X_{t}^{x,r},\widetilde X_t^{y,r})\in d(x_1, y_1)\right) \leq \Prob(\widetilde X_{s+t}^{x,r} \neq \widetilde X_{s+t}^{y,r}).
\end{equation}
\item Finally, if $(\widetilde X_{t},\widetilde X_t)_{t\geq 0}$ is a Markovian maximal coupling of $(X_t)_{t\geq 0}$, i.e. the conditions in (ii) hold and for any $x, y\in \R^d,$ $r,t\geq 0$ the equality $\dtv(X_t^{x,r}, X_t^{{y,r}}) = \Prob(\widetilde X_t^{{x,r}} \neq \widetilde X_t^{{y,r}})$ holds, then
\begin{equation} \label{dtv-markovian-maximal-eq}
\dtv(X_{t+s}^{x,r},X_{t+s}^{y,r}) = \int \dtv(X_s^{x_1,t+r}, X_s^{{y_1,t+r}})\, \Prob\left((\widetilde X_{t}^{x,r},\widetilde X_t^{y,r})\in d(x_1, y_1)\right).
\end{equation}
\end{enumerate}
\end{theorem}
\begin{proof}
Let $(X_t)_{t\geq 0}$ be a Markov process and $(\widetilde X^{x,r}_{t},\widetilde X^{y,r}_t)\in \cC(X_t^{x,r};X_t^{y,r})$. Then, by the Markov property, 
\begin{equation}
\begin{split}
\dtv(X_{t+s}^{x,r}, X_{t+s}^{y,r}) & = \sup_{A\in \Bskript(\R^d)} \left| \Prob( X_{t+s}^{x,r} \in A) - \Prob( X_{t+s}^{y,r} \in A)\right|\\
 & = \sup_{A\in \Bskript(\R^d)} \left| \int (\Prob( X_{s}^{x_1,t+r} \in A) - \Prob( X_{s}^{y_1,t+r} \in A))\, \Prob((\widetilde X_t^{x,r}, \widetilde X_t^{y,r}) \in d(x_1,y_1))\right|\\
 & \leq \int \sup_{A\in \Bskript(\R^d)} \left|\Prob( X_{s}^{x_1,t+r} \in A) - \Prob( X_{s}^{y_1,t+r} \in A)\right|\, \Prob((\widetilde  X_t^{x,r},\widetilde X_t^{y,r}) \in d(x_1,y_1))\\
& =  \int \dtv(X_s^{{x_1,t+r}}, X_s^{{y_1,t+r}})\, \Prob\left((\widetilde X_{t}^{x,r},\widetilde X_t^{y,r})\in d(x_1, y_1)\right),
\end{split}
\end{equation}
i.e. \eqref{dtv-immersed-lower-bound} holds. 

The inequality \eqref{dtv-markovian-lower-bound} can be proved using the coupling time (cf. \cite[Lemma 3.3]{Kuwa2009}) or directly using the Markov property and \eqref{eq:couplingtime}:
\begin{equation} \label{eq:dtv-proof-lower-bound}
\begin{split}
\Prob(\widetilde X_{t+s}^{x,r} \neq \widetilde X_{t+s}^{y,r}) & = \int \Prob(\widetilde X_{s}^{x_1,t+r} \neq \widetilde X_{s}^{y_1,t+r}) \, \Prob\left((\widetilde X_t^{x,r},\widetilde X_t^{y,r})\in d(x_1, y_1)\right)\\
&\geq \int \dtv(X_{s}^{{x_1,t+r}}, X_{s}^{{y_1,t+r}})\, \Prob\left((\widetilde X_t^{x,r},\widetilde X_t^{y,r})\in d(x_1, y_1)\right).
\end{split}
\end{equation}

From the first line of \eqref{eq:dtv-proof-lower-bound} follows \eqref{dtv-markovian-maximal-eq} directly using the definition of maximal coupling \eqref{tv_eq}. 
\end{proof}

In fact \eqref{dtv-markovian-maximal-eq} yields a characterization of Markovian maximal couplings.

\begin{corollary}Let $(X_t)_{t\geq 0}$ be a Markov process. For any $x,$ $y$ and $r\geq 0$ the mapping $t\mapsto \dtv(X^{x,r}_t,X^{y,r}_t)$ be continuous and the process $(\widetilde X^{x,r}_{t},\widetilde X^{y,r}_t)_{t\geq 0}$ be a Markovian coupling of $(X^{x,r}_t)_{t\geq 0}$ and $(X^{y,r}_t)_{t\geq 0}$ satisfying \eqref{dtv-markovian-maximal-eq}. Then it is a Markovian maximal coupling. 
\end{corollary}
\begin{proof}
Letting $s$ tend to $0$ in \eqref{dtv-markovian-maximal-eq} yields
\begin{equation}
\dtv(X_{t}^{x,r},X_{t}^{y,r}) = \Prob (\widetilde X_{t}^{x,r} \neq \widetilde X_t^{y,r}).
\end{equation} 
Thus the coupling is maximal.
\end{proof}

Furthermore, for time-homogeneous processes one gets the following.

\begin{theorem} \label{thm:max-coupling-time-homo} Let $(X_t)_{t\geq 0}$ be a time-homogeneous Markov process. If $(\widetilde X_{t},\widetilde X_t)_{t\geq 0}$ is a Markovian maximal coupling of $(X_t)_{t\geq 0}$ then 
\begin{equation} \label{dtv-markovian-maximal-eq-th-tih}
\dtv(X_{t+s}^x,X_{t+s}^y) = \int \dtv(X_s^{x_1}, X_s^{{y_1}})\, \Prob\left((\widetilde X_{t}^{x,r},\widetilde X_t^{y,r})\in d(x_1, y_1)\right) \text{ for all $r\geq 0$,}
\end{equation}
and if it is unique, i.e. \eqref{dtv-markovian-maximal-eq-th-tih} determines the distribution of $(\widetilde X_{t}^{x,r},\widetilde X_t^{y,r})$ uniquely, then the Markovian maximal coupling is time-homogeneous and
\begin{equation} \label{dtv-markovian-maximal-eq-th}
\dtv(X_{t+s}^x,X_{t+s}^y) = \int \dtv(X_s^{x_1}, X_s^{{y_1}})\, \Prob\left((\widetilde X_{t}^{x},\widetilde X_t^{y})\in d(x_1, y_1)\right).
\end{equation}
\end{theorem}
\begin{proof} Let $(\widetilde X_{t},\widetilde X_t)_{t\geq 0}$ be a Markovian maximal coupling of $(X_t)_{t\geq 0}$. Since the marginals are time-homogeneous 
\begin{equation}
\dtv(X_{t}^{x,r},X_{t}^{y,r}) = \dtv(X_{t}^{x,0},X_{t}^{y,0}) = \dtv(X_{t}^{x},X_{t}^{y})
\end{equation}
holds for all $t,r\geq 0$ and  $x,y\in \R^d$ and thus \eqref{dtv-markovian-maximal-eq} simplifies to  \eqref{dtv-markovian-maximal-eq-th-tih}. If the coupling is unique, the distribution of $(\widetilde X_{t}^{x,r},\widetilde X_t^{y,r})$ cannot depend on $r$. Thus \eqref{dtv-markovian-maximal-eq-th-tih} becomes \eqref{dtv-markovian-maximal-eq-th}.
\end{proof}

Note that for a distance $\rho$
\begin{equation} \label{eq:wasserst}
 W_{\rho}(X,Y) := \inf_{(\widetilde X,\widetilde Y) \in \cC(X,Y)}\int \rho(x_1, y_1)\, \Prob\left((\widetilde X,\widetilde Y)\in d(x_1, y_1)\right)
\end{equation}
defines a Wasserstein (Vasershtein) type distance, cf. \cite[Chapter 7]{Vill2003}. In particular, with $\rho(x,y) = 1$ if $x\neq y$ and $\rho(x,y)=0$ if $x=y$ one gets $\dtv(X,Y) = W_\rho(X,Y)$, cf. \eqref{tv_inf}. This could be used to rewrite the proof of Theorem \ref{thm:mmc-characterization}. But note that in this case the minimizer in \eqref{eq:wasserst} is not unique. In general, conditions for the uniqueness of the minimizer  are known, e.g. if $\rho$ is of the form $\rho(x,y)=f(|x-y|)$ for a strictly concave function $f,$ cf. \cite{GangMcCa1996}. In \eqref{dtv-immersed-lower-bound} a distance of Wasserstein type appears and a lower bound is provided, which is attained for the Markovian maximal coupling (if it exists). This is the key idea for the uniqueness result in the next section.

\section{Maximal coupling of subordinated Brownian motions} \label{sec:maxsBM}

For readers not familiar with L\'evy processes and subordination the basic definitions and results can be found in the appendix (Section \ref{appendix}). In \cite{BoetSchiWang2011a} subordinated Brownian motions were coupled by subordinating reflection coupled Brownian motions. The coupling therein was used to obtain estimates of the total variation distance of the distributions of the processes as time $t \to \infty.$ The key estimate is (\cite[p. 1210]{BoetSchiWang2011a}): 
\begin{equation} \label{eq:tv-subordinator-bound}
\dtv(L_t^x,L_t^y) \leq \int_{[0,\infty)} \Prob(T(\widetilde B_.^x,\widetilde B_.^y)> s)\, \Prob(S_t \in ds)
\end{equation}
where $(L_t^z)_{t\geq 0}$ is a $\R^d$-valued subordinated Brownian motion started at $z$, $(\widetilde B_t^x,\widetilde B_t^y)_{t\geq 0}$ is the reflection coupling of $\R^d$-valued Brownian motions (started at $x$ and $y$, respectively) and $(S_t)_{t\geq 0}$ is a subordinator such that $L_.^z \eqd B_{S_.}^z$ for $z\in\{x,y\}$. We call $(\widetilde B_{S_t}^x,\widetilde B_{S_t}^y)_{t\geq 0}$ \textbf{subordinated reflection coupled Brownian motion}. Note that the subordinated Brownian motion has transition probabilities of the form 
\begin{equation}
\Prob(B_{S_t}^z-B_{S_s}^z \in dx) = p_{t-s}(|x|) dx + \Prob(S_{t-s}=0) \delta_0(dx)
\end{equation} where $p_t(|\cdot|)$ is a (sub probability) transition density and $p_t$ is monotone on $(0,\infty)$ since 
\begin{equation} \label{eq:sub-bm-density}
p_t(r) = \int_{(0,\infty)} \frac{1}{(2\pi)^\frac{d-1}{2}} \frac{1}{\sqrt{2\pi s}} e^{-\frac{r^2}{2s}}\, \Prob(S_t\in ds)
\end{equation} has derivative $p_t'(r) < 0$ for $r>0$.   

By construction $(\widetilde B_{S_.}^x,\widetilde B_{S_.}^y)$ is a Markovian coupling. The next Lemma shows that in \eqref{eq:tv-subordinator-bound} in fact equality holds and it implies that for subordinated Brownian motions the subordination of reflection coupled Brownian motions is a maximal coupling (by \eqref{eq:couplingtime} and \eqref{eq:def:maxcoupling-processes}). 

\begin{lemma}
In the above setting and notation we have
\begin{equation} \label{eq:tv-subordinated}
\dtv(L_t^x,L_t^y) = \Prob(S_t =0)(1-\delta_{x}(y)) + 2 (2\pi)^\frac{d-1}{2} \int_0^\frac{|x-y|}{2} p_t(r)\, dr = \Prob(T(\widetilde B_{S_.}^x,\widetilde B_{S_.}^y) > t). 
\end{equation}
\end{lemma}
\begin{proof} First note that for $x=y$ the equality holds since each term is equal to zero. Thus it remains to consider the case $x\neq y$:

We need three technical observations: 1. If $X$ and $Y$ have distributions of the form $f(z)dz + c \delta_a(dz)$ and $g(z)dz + c \delta_b(dz)$ for $c\in(0,1)$ and some $a,b$ then
\begin{equation}
\dtv(X,Y) = (1-c) - \int f(z)\land g(z)\, dz + c(1-\delta_a(b))
\end{equation}
and for $a\neq b$ this becomes
\begin{equation} \label{dtv-subprob}
\dtv(X,Y) = 1 - \int f(z)\land g(z)\, dz.
\end{equation}
2. In the setting of subprobability densities Equation \eqref{mirror-coupling-density} remains valid, if one replaces the leading one by the total mass of the subprobability density. \\
3. Observe that
\begin{equation}
 \int_{\R^{d-1}} p_t\left(\sqrt{ z_1^2 + z_2^2 + \ldots + z_d^2}\right)\,\lambda^{d-1}(dz_2\ldots z_d) =  (2\pi)^\frac{d-1}{2} p_t(z_1).
\end{equation}
Thus by \eqref{dtv-subprob}, \eqref{mirror-coupling-density} and \eqref{bm-couplingtime} we find
\begin{equation} \label{eq:tv-subordinated}
\begin{split}
\dtv(L_t^x,L_t^y) &= \Prob(S_t = 0) + 2 \int_0^\frac{|x-y|}{2} \int_{\R^{d-1}} p_t\left(\sqrt{ z_1^2 + z_2^2 + \ldots + z_d^2}\right)\,\lambda^{d-1}(dz_2\ldots z_d)\, \lambda^1(dz_1)\\
&= \Prob(S_t = 0) +  2 \int_0^\frac{|x-y|}{2} (2\pi)^\frac{d-1}{2} p_t\left(z_1\right)\, \lambda^1(dz_1)\\
&= \Prob(S_t = 0) + \int_0^\frac{|x-y|}{2} 2 \int_{(0,\infty)} g_s(r)\, \Prob(S_t\in ds)\, dr\\
&= \int_{[0,\infty)}  2 \int_0^\frac{|x-y|}{2} g_s(r)\, dr\,\Prob(S_t\in ds)\\
&= \int_{[0,\infty)} \Prob(T(\widetilde B_.^x,\widetilde B_.^y) > s)\,\Prob(S_t\in ds)\\
&= \Prob(T(\widetilde B_{S_.}^x,\widetilde B_{S_.}^y) > t). \qedhere
\end{split}
\end{equation}
\end{proof}

In particular the above shows that $\dtv(B_s^x,B_s^y) = 2 \int_0^\frac{|x-y|}{2} g_{s}(z)\, dz$, where $g_{s}$ is the density of the normal distribution with mean 0 and variance $s$, cf. \eqref{bm-couplingtime}. This together with the uniqueness condition for \eqref{eq:wasserst} was used in \cite{HsuStur2003} to show that for Brownian motion the mirror coupling is the unique Markovian maximal coupling. In our setting it is a special case of the following result.

\begin{theorem} \label{thm:mmc-subordinated-bm-uniqueness} For subordinated Brownian motion the subordinated reflection coupled Brownian motion is the unique Markovian maximal coupling.
\end{theorem}
\begin{proof}
By \eqref{eq:tv-subordinated} we have $\dtv(L_t^x,L_t^y) = \Prob(S_t =0)(1-\delta_{x}(y)) + 2 (2\pi)^\frac{d-1}{2} \int_0^\frac{|x-y|}{2} p_t(r)\, dr$ and this is a strictly concave function of $|x-y|$ since for $x>0$
\begin{equation}
\frac{d^2}{dx^2} \left(2\int_0^{\frac{x}{2}} p_t(r)\, dr\right) = \frac{d}{dx} \left(p_t\left(\frac{x}{2}\right)\right) = \frac{1}{2} p_t'\left(\frac{x}{2}\right) < 0,
\end{equation}  cf. \eqref{eq:sub-bm-density}. Thus \eqref{eq:wasserst} has a unique solution for this distance. Hence the given solution of \eqref{dtv-markovian-maximal-eq-th-tih} is unique.
\end{proof}

A different approach to maximal coupled subordinated Brownian motion will be given in Theorem \ref{thm:mmc-sub-bm-representations}.

\section{Feller processes and generators of Markovian couplings}

For basics about Feller processes we refer again to the appendix (Section \ref{appendix}). We start with conditions for a Markovian coupling to be a Feller process.  

\begin{lemma} \label{thm:coupling-of-feller} Suppose $(X_t)_{t\geq 0}$ and $(Y_t)_{t\geq 0}$ are Feller processes on $\R^d$ and $(X_t,Y_t)_{t\geq 0}$ is a time-homogeneous Markov process, then the semigroup $(T_t)_{t\geq 0}$ on $B_b(\R^{2d})$ corresponding to $(X_t,Y_t)_{t\geq 0}$ satisfies
\begin{equation} \label{eq:Binftyinvariant}
T_t: B_\infty \to B_\infty
\end{equation}
where $B_\infty$ are the bounded measurable functions vanishing at $\infty$. Furthermore, if the semigroup satisfies $T_t:C_\infty \to C$ for each $t\geq 0$ (using \eqref{eq:Binftyinvariant} it thus satisfies the Feller property) then
\begin{equation}
(T_t)_{t\geq 0} \text{ is strongly continuous } 
\end{equation}
and thus $(X_t,Y_t)_{t\geq 0}$ is a Feller process.
\end{lemma}
\begin{proof}
Since $(X_t,Y_t)_{t\geq 0}$ is a time-homogeneous Markov process a corresponding semigroup $(T_t)_{t\geq 0}$, given by
\begin{equation}
T_tf(x,y) = \E(f(X_t^x,Y_t^y)) \text{ for $f$ bounded and measurable,}
\end{equation}
exists. Let $f\in B_\infty(\R^{2d}).$ For $\varepsilon>0$ exists $R_\varepsilon>0$ such that $|f(x)|<\varepsilon$ for $|x|\geq R_\varepsilon$ and $\tilde f_\varepsilon\in C_\infty(\R^{d})$ such that $\tilde f_\varepsilon \geq \One_{B^d_{R_\varepsilon}}$, where $B^d_R$ is the ball with center 0 and radius $R$ in $\R^d$. Thus

\begin{equation}
\begin{split}
\E(|f(X_t^x,Y_t^y)|) &\leq \varepsilon + \|f\|_\infty \E(\One_{B^{2d}_{R_\varepsilon}}(X_t^x,Y_t^y))\\
&\leq \varepsilon + \|f\|_\infty\Prob(|X_t^x|\leq R_\varepsilon, |Y_t^y|\leq R_\varepsilon)\\
&\leq \varepsilon+ \|f\|_\infty\Prob(|X_t^x|\leq R_\varepsilon) + \|f\|_\infty\Prob(|Y_t^y|\leq R_\varepsilon)\\
&\leq \varepsilon + \|f\|_\infty\widetilde T_t\tilde f_\varepsilon(x) + \|f\|_\infty\widetilde S_t\tilde f_\varepsilon(y)
\end{split}
\end{equation} 
where $\widetilde T_t$ and $\widetilde S_t$ are the Feller semigroups of $X_t$ and $Y_t$, respectively. Thus 
\begin{equation}
T_tf(x,y)\xrightarrow{|(x,y)|\to \infty}0.
\end{equation}

Assume $T_t:C_\infty \to C$. Now \eqref{eq:Binftyinvariant} implies the Feller property. Since $(X_t)_{t\geq 0}$ and $(Y_t)_{t\geq 0}$ are Feller processes they are stochastically continuous (in time) and so $(X_t^x,Y_t^y)$ is stochastically continuous at $t = 0$. Which implies $T_t(f(x,y))-f(x,y) \xrightarrow{t\to 0} 0$ for $(x,y)\in\R^{2d}$ and $f\in C_\infty$. In the current setting this is equivalent to the strong continuity of $T_t$, see \cite[Lemma 1.4]{BoetSchiWang2013}.
\end{proof}

The following examples show that not all Markovian couplings of Feller processes are Feller processes.

\begin{example} \label{ex:coupling-time-inhomo-not-strong-mp}
\begin{enumerate}
\item A time-inhomogeneous Markovian coupling of Feller processes can constructed as follows: Start a Brownian motion at $x$ and let an identical copy start in $y$, let them both run until time 1, afterwards let them continue as independent Brownian motions. I.e., they are syncronized coupled up to time 1 and afterwards independent. The marginals are Brownian motions (and thus in particular Feller processes), but the joint process is a time-inhomogeneous Markov process.  
\item A Markovian coupling of Feller processes can be a time-homogeneous Markov process, which is not a strong Markov process. Let $(L_t)_{t\geq 0}$ be a symmetric L\'evy process on $\R$ (starting in 0) and consider the process
\begin{equation}
(X_t^x,Y_t^y):= \begin{cases} (x+L_t,y-L_t) &\text{ for } x\neq y, \\ (x+L_t,y+L_t) &\text{ for } x=y. \end{cases} 
\end{equation}
Then the components are Feller processes and the joint process is a Markov process. But it is not a strong Markov process, and thus not a Feller process. The proof is analogous to the proof that $B_t^{x}\One_{\R\backslash\{0\}}(x)$ is not a strong Markov process, e.g. \cite[Counterexample 20.8.]{Stoy1997}.
\end{enumerate}
\end{example}

The processes we considered in the previous section, however, are Feller processes.

\begin{theorem}\label{thm:rcbm-feller} The reflection coupling of Brownian motion is a Feller process.
\end{theorem}
\begin{proof}
For each $x$ and $y$ let $(B_t^x,B_t^y)_{t\geq 0}$ be reflection coupled Brownian motion, and hence mirror coupled Brownian motion. Denote by $g_{t,x}$ the density of $B_t^x$. For $f\in C_\infty$  
\begin{equation}
\E(f(B_t^x,B_t^y)) =  \int f(z,z) (g_{t,x}(z)\land g_{t,y}(z))\,dz + \int f(z,R_{x,y}z) (g_{t,x}(z)-g_{t,y}(z))^+\,dz
\end{equation}
and the integrands are bounded and continuous. For the continuity with respect to $(x,y)$ note that $(x,y)\mapsto g_{t,x}(z) \land g_{t,y}(z)$, $(x,y)\mapsto (g_{t,x}(z)-g_{t,y}(z))^+$ are continuous and the latter vanishes for $x=y$, and $(x,y) \mapsto R_{x,y}z$ is continuous for $x\neq y$. Thus dominated convergence implies that $(x,y) \mapsto \E(f(B_t^x,B_t^y))$ is continuous. Now Lemma \ref{thm:coupling-of-feller} implies the statement.
\end{proof}

By \cite[Lemma 4.5]{BoetSchiWang2013} the next result follows. 
\begin{corollary} \label{thm:srcbm-feller}
Subordinated reflection coupled Brownian motion is a Feller process.
\end{corollary}

Note that these Feller processes are clearly not irreducible, since every path lives eventually only on $\{(x,x) \,|\, x\in\R^d\} \subset \R^{2d}.$

Now we calculate the generators of couplings which are Feller processes. 

\begin{example}
\begin{enumerate}
\item For a one-dimensional Brownian motion the synchronized coupling $(B_t^0+x,B_t^0+y)$ has the generator
\begin{equation}
A f(x,y) = \frac{1}{2} \sum_{j,k=1}^2 \partial_j\partial_k f(x,y) = \frac{1}{2} \Delta f(x,y) +  \partial_x\partial_y f(x,y).
\end{equation}
Note that the covariance matrix is given by $\left(\begin{array}{cc}
1 & 1\\
1& 1\\
\end{array}\right).$

\item For a one-dimensional Brownian motion the coupling $(B_t^0+x,-B_t^0+y)$ has generator
\begin{equation}
A f(x,y) = \frac{1}{2} \sum_{j,k=1}^2 \partial_j\partial_k f(x,y) (-1)^{k+j} = \frac{1}{2} \Delta f(x,y) - \partial_x\partial_y f(x,y).
\end{equation}
Note that the covariance matrix is given by $\left(\begin{array}{cc}
1 & -1\\
-1& 1\\
\end{array}\right).$
\item For a one-dimensional Brownian motion the reflection coupling $(B_t^x,(R_{x,y}^\tau B^x)_t)_{t\geq 0}$ has the (extended pointwise, cf. Section \ref{appendix}) generator

\begin{equation}
A f(x,y) = \frac{1}{2} \Delta f(x,y) + (1-2\One_{\R^d\backslash\{0\}}(x-y)) \partial_x\partial_y f(x,y).
\end{equation}
This is a special case of the following result. 
\end{enumerate}
\end{example}
\begin{theorem}
Let $B_t^0$ be a Brownian motion on $\R^d$. Then the (extended pointwise) generator of the reflection coupling is on the twice continuously differentiable functions given by
\begin{equation} \label{eq:rcbm-generator}
Af(x,y) =  \frac{1}{2}\Delta f(x,y) +  \sum_{i=1}^d \sum_{k=1}^d \left(\delta_{i,k} - 2 \frac{(x_i-y_i)(x_k-y_k)}{|x-y|^2}\right) \partial_{x_i}\partial_{y_k} f(x,y) 
\end{equation}
with $\delta_{i,k} = \begin{cases}
1 &, i=k\\0 &, i\neq k
\end{cases}$ and the convention $\frac{0}{0} :=0$.
\end{theorem}

\begin{proof}
Let $x,y,z\in \R^d$, $h(z) := (z,R_{x,y} z)$ and $f \in C_c^\infty(\R^{2d})$. Then the (extended pointwise) generator of $((B_t^x, R_{x,y}^\tau B_t^x))_{t\geq 0}$ is calculated via
\begin{equation}
\lim_{t\to 0} \E\frac{f(B_t^x, R_{x,y}^\tau B_t^x) - f(x,y)}{t} = \lim_{t\to 0} \E\frac{f(B_t^x, R_{x,y} B_t^x) - f(x,R_{x,y} x)}{t}
\end{equation}
and
\begin{equation}
\begin{split}
\lim_{t\to 0} \E\frac{f(B_t^z, R_{x,y} B_t^z) - f(z,R_{x,y} z)}{t} = \lim_{t\to 0} \E\frac{f(h(B_t^0+z)) - f(h(z))}{t} = \frac{1}{2} \Delta_z (f\circ h)(z).
\end{split}
\end{equation}
Now the chain rule (with a tedious calculation) and setting $z=x$ yields the result.
\end{proof}

\begin{remark}
\begin{enumerate}
\item If $f\in C_\infty (\R^{2d})$ and $Af\in C_\infty(\R^{2d})$ in \eqref{eq:rcbm-generator}, then $f$ is also in the domain of the generator \cite[p. 23]{BoetSchiWang2013}. The generator coincides with the reflection coupling generator for diffusions as given in \cite[Example 5.2]{Chen1994a} when replacing therein $\sigma$ with the identity matrix. 
\item Obviously, since the coefficients in  \eqref{eq:rcbm-generator} are discontinuous at $x=y$, the test functions are not (and cannot, cf. \cite[Lemma 2.28]{BoetSchiWang2013}) be in the domain of the generator. Thus in particular they are not a core.
\item Note that for $R_{x,y}$ (and not $R_{x,y}^\tau$) we have
 \begin{equation}
Cov((B_{t}^x)_{(i)},(R_{x,y}B_{t}^x)_{(k)}) = \delta_{i,k} - 2 \frac{(x_i-y_i)(x_k-y_k)}{|x-y|^2}.
\end{equation}
\item Note that for reflection coupled Brownian motion $(\widetilde B_t^x,\widetilde B_t^y)$ one has
\begin{equation} \label{eq:norm-reflection}
|(\widetilde B_t^x -x, \widetilde B_t^y -y)| = \sqrt{2}|B_t^0|.
\end{equation}

\end{enumerate}
\end{remark}

\begin{theorem} \label{thm:mmc-sub-bm-representations}
For fixed $x,y\in \R^d$, a subordinator $(S_t)_{t\geq 0}$ and a Brownian motion $(B_t)_{t\geq 0}$ on $\R^d$ the 
\begin{enumerate}
\item[(i)] subordinated reflection coupled Brownian motion, i.e., $(B_{S_t}^x,(R_{x,y}^\tau B^x)_{S_t})_{t\geq 0}$,
\end{enumerate} 
is the unique Markovian version of
\begin{itemize}
\item[(ii)] mirror coupled subordinated Brownian motion, i.e., $(X_.,Y_.)\in \cC(B_{S_.}^x;B_{S_.}^y)$ where $(X_t,Y_t)$ are mirror coupled for each $t\geq 0$. (Note this does not give an explicit construction, it is just a distributional statement).
\end{itemize}
Moreover, the process given in (i) is a Feller process whose corresponding family of L\'evy measures $(N((x,y),.))_{(x,y)\in\R^{2d}}$ is determined by the state space dependent mirror coupling of the Levy measures of the subordinated Brownian motion:
\begin{equation} \label{eq:mmc-feller-generator}
N((x,y),A) = \int \One_A(z-x,z-y) (n(z-x) \land n(z-y))\, dz + \int \One_A(z-x,  R_{x,y}(z-y)) (n(z-x) - n(z-y))^+\, dz
\end{equation} 
where $n$ is the density of the L\'evy measure of $B_{S_.}$ and $A\in \Bskript(\R^{2d}\backslash\{0\}).$

\end{theorem}
\begin{proof}
Given the process in (i), by Theorem \ref{thm:mmc-subordinated-bm-uniqueness} it is a unique Markovian maximal coupling. To see that it satisfies (ii) note that by construction either:\\
a) $B_{S_t}^x$ and $(R_{x,y}^\tau B^x)_{S_t}$ coincide with maximal probability (since it is a maximal coupling) or\\
b) $S_t(\omega) <\tau(\omega)$ and thus by definition $B_{S_t(\omega)}^x(\omega)$ is the mirror image of $R_{x,y} B^x_{S_t(\omega)}(\omega)$, i.e. \eqref{mirror-coupling-mirror} holds.\\
Thus the process satisfies (ii).

Let a Markov process be mirror coupled subordinated Brownian motion, i.e. it satisfies (ii). The marginal distributions of the processes in (i) and (ii) coincide since both are in $\cC(B_{S_.}^x;B_{S_.}^y)$. Also the joint distributions for each fixed time coincide. Thus, since it is a Markov process, the processes coincide in distribution. 

Finally we look the family of L\'evy measures of the subordinated reflection coupled Brownian motion.
Let $(\widetilde B_t^x,\widetilde B_t^y)_{t\geq 0}$ be reflection coupled Brownian motion started in $(x,y)$. Let $g_{t,z}$ be the density of $\R^d$-valued Brownian motion at time $t$ with starting point $z$, i.e., $g_{t,z}(x) dx = \Prob(B^z_t\in dx)= g_{t,0}(x-z) dx.$ Then for $A\in \Bskript(\R^{2d})$ 
\begin{equation}
\Prob( (\widetilde B_t^x,\widetilde B_t^y) \in A) = \int \One_A(z,z) (g_{t,x}(z) \land g_{t,y}(z)) \, dz + \int \One_A(z, R_{x,y}z) (g_{t,x}(z) - g_{t,y}(z))^+\, dz.
\end{equation}
Note that for a relative compact set $A\in \Bskript(\R^{2d}\backslash\{0\})$ there exists an $\varepsilon>0$ such that $A \subset \{ z \in\R^{2d} \,|\, |z|\geq \varepsilon \}$. Thus with the Markov inequality and \eqref{eq:norm-reflection}
\begin{equation}
\Prob((\widetilde B_t^x-x,\widetilde B_t^y-y) \in A)  \leq \Prob(|(\widetilde B_t^x-x,\widetilde B_t^y-y)| \geq \varepsilon) \leq \frac{8\E(|\widetilde B_t^0|^6)}{\varepsilon^6} \leq c t^3
\end{equation}
for some constant $c$. Now let $\nu$ be the L\'evy measure of $S_.$ and $n$ be the density of the L\'evy measure of $B_{S_.},$ that is $n(z) = \int_0^\infty g_{s,0}(z)\,\nu(ds)$. Then \eqref{eq:calc-levymeasure} yields
\begin{equation}
\begin{split}
&\frac{1}{t}  \int_{(0,\infty)} \Prob( (\widetilde B_s^x-x,\widetilde B_s^y-y) \in A)\, \Prob(S_t \in ds) \\
&\xrightarrow{t\to 0} \int_{(0,\infty)} \Prob( (\widetilde B_s^{x}-x,\widetilde B_s^y-y) \in A) \,\nu(ds)\\
&=\int \int_{(0,\infty)} \One_A(z-x,z-y) (g_{s,x}(z) \land g_{s,y}(z)) \,\nu(ds) dz + \int \int_{(0,\infty)} \One_A(z-x, R_{x,y}(z-y)) (g_{s,x}(z) - g_{s,y}(z))^+\, \nu(ds) dz \\
&=\int \One_A(z-x,z-y) (n(z-x) \land n(z-y))\, dz + \int \One_A(z-x,  R_{x,y}(z-y)) (n(z-x) - n(z-y))^+\, dz.
\end{split}
\end{equation}
Here we used for the last equality that the rotational symmetry (and directional monotonicity, compare \eqref{eq:sub-bm-density}) of the densities implies
\begin{equation}
\begin{split}
z\in H_{x,y} \quad \Leftrightarrow \quad&\exists s>0:\ g_{s,0}(z-x) = g_{s,0}(z-y) \\
\Leftrightarrow \quad& \forall s>0:\ g_{s,0}(z-x) = g_{s,0}(z-y) \\
\Leftrightarrow \quad& n(z-x) = n(z-y)
\end{split}
\end{equation}
and analogous for '$>$' and '$<$'. 
\end{proof}

Thus it is natural to conjecture that for a symmetric L\'evy process with L\'evy density $n$ a maximal coupled process is defined by \eqref{eq:mmc-feller-generator}. But note, since the test functions are not in the domain of the generator the usual methods of construction of such processes (cf. \cite{BoetSchiWang2013}) can not be applied.

The presented method works only for symmetric L\'evy processes. A natural question is, under which conditions does there exist a Markovian maximal coupling for more general Feller processes (which are truly state space dependent, or non symmetric). The above results seem to indicate, that at least the L\'evy measures should be maximally coupled. But, as we have noted before, this fixes only part of the joint L\'evy measure - the remaining part is an open problem. Moreover, without the test functions in the domain of the generator also the construction, given the state space dependent L\'evy triple, seems to call for new techniques.

\section{Appendix - Feller processes, L\'evy processes, subordination} \label{appendix}
The monograph \cite{BoetSchiWang2013} is the general reference for this section, comprehensive details about L\'evy processes and subordinators can be found in \cite{Sato99}.

A time homogeneous (strong) Markov process $(X_t)_{t\geq 0}$ on $\R^d$ is a \textbf{Feller process}, if the corresponding semigroup $(T_t)_{t\geq 0}$ on $C_\infty:=\{f:\R^d \to \R\ :\ f \text{ is continuous and }\lim_{|x|\to \infty} f(x) = 0\}$ given by
\begin{equation}
T_tf(x):= \E f(X_t^x)
\end{equation} 
is a \textbf{Feller semigroup}, i.e., it is a strongly continuous ($\lim_{t\to 0}\|T_t f-f\|_\infty =0$) contraction ($\|T_tf\|_\infty \leq \|f\|_\infty$)  semigroup ($T_0 =id,\ T_{t+s} = T_t\circ T_s,\ T_t:C_\infty \to C_\infty$) which is positivity preserving ($f\geq 0 \,\Rightarrow\,T_tf\geq 0$). The linear operator $A$ defined as the strong derivative of $T_t$ at $t=0$ is called the \textbf{generator}, i.e.,
\begin{equation} \label{def:generator}
\lim_{t\to 0} \left\| \frac{T_t f-f}{t} - Af\right\|_\infty = 0.
\end{equation} 
The domain of the generator are all $f\in C_\infty$ for which the above limit exists. We call $\hat A$ the \textbf{extended pointwise generator}, if the limit in \eqref{def:generator} is considered pointwise (instead of uniform). The domain of this operator is the set of all functions in $C_\infty$ for which all pointwise limits exist. Note that in this case $x\mapsto \hat Af(x)$ might not be continuous.

A spatially homogeneous Feller process is a \textbf{L\'evy process}, i.e., it is a stochastically continuous process with stationary and independent increments. In particular, Brownian motion is a L\'evy process. A \textbf{subordinator} is a L\'evy process with non decreasing paths.
For L\'evy processes the characteristic function of an increment $X_{t+s}-X_s$ is of the form $\xi \mapsto e^{-t\psi(\xi)}$ where
\begin{equation} \label{eq:lkf}
\psi(\xi) = \psi(0) - il\xi + \frac{1}{2} \xi Q \xi + \int_{\R^d\backslash\{0\}} (1-e^{iy\xi} + iy\xi \One_{(|y|\leq 1)})\nu(dy)
\end{equation}
with killing term $\psi(0)\geq 0$, drift vector $l\in\R^d$, Brownian coefficient $Q\in\R^{d\times d}$ and L\'evy measure $\nu$ satisfying $\int_{\R^d\backslash\{0\}} |y|^2\land 1 \, \nu(dy) <\infty$.  $(l,Q,\nu)$ is called \textbf{L\'evy triple} and it determines the process uniquely (in distribution).

We present a result which allows to calculate the L\'evy measure based on the transition probabilities, cf. \cite{Figu2008}. 
  
\begin{lemma}
Let $L_t$ be an $\R^d$-valued L\'evy process (starting in 0) with L\'evy measure $\mu$ and $f:\R^d\to \R$ be continuous, bounded and with $|f(x)| \leq c |x|^{2+\varepsilon}$ for some $c,\varepsilon>0$ and all $x\in B^d_\delta:=\{x\in\R^d\ :\ |x|\leq \delta\}$ for some $\delta>0$. Then
\begin{equation} \label{eq:calc-levymeasure}
\frac{\int_{\R} f(x)\, \Prob(L_t\in dx)}{t} \xrightarrow{t\to 0} \int_{\R} f(x) \mu(dx).
\end{equation}
\end{lemma}
\begin{proof}
Let $L_t$ and $f$ be as required, $\delta>0$ and $\chi_\delta\in C_c(\R^d)$ with $\One_{B^d_{\delta/2}} \leq \chi_\delta \leq \One_{B^d_\delta}$. Then  $(1-\chi_\delta)f\in C_c(\R^d\backslash\{0\})$ and \cite[Lemma 2.16]{BoetSchiWang2013} and dominated convergence imply
\begin{equation} \label{eq:levymeasure-proof-1}
\frac{1}{t}\int (1-\chi_\delta(x)) f(x)\, \Prob(L_t\in dx) \xrightarrow{t\to 0} \int(1-\chi_\delta(x))f(x)\,\mu(dx) \xrightarrow{\delta\to 0} \int f(x)\, \mu(dx).
\end{equation}
By the truncation inequality (cf. \cite{Usha2000}) and the fact that the characteristic exponent of a L\'evy process growth at most quadratic (cf. \cite[Eq. (2.11)]{BoetSchiWang2013}) we get with $m:= \sup_{x\in B^d_\delta} |f(x)|$
\begin{equation} \label{eq:levymeasure-proof-2}
\begin{split}
\frac{1}{t}\int \chi_\delta(x) f(x)\, \Prob(L_t\in dx) &=  \frac{1}{t}\int_0^\infty \Prob(\chi_\delta(L_t) f(L_t) \geq r)\, dr\\
&\leq \frac{1}{t} \int_0^m \Prob(c|L_t|^{2+\varepsilon}\geq r)\,dr\\
&\leq \frac{1}{t}   \int_0^m c'\int_{|\xi|\leq \left(\frac{c}{r}\right)^{\frac{1}{2+\varepsilon}}} (1- \mathrm{Re}(  e^{-t\psi(\xi)}))\,d\xi\,dr\\
&\leq \frac{1}{t} c'  \int_0^m t \sup_{|\xi|\leq \left(\frac{c}{r}\right)^{\frac{1}{2+\varepsilon}}} |\psi(\xi)|\, dr\\
&\leq c'' \int_0^m \left(\frac{c}{r}\right)^{\frac{2}{2+\varepsilon}}\, dr \xrightarrow{\delta \to 0} 0,
\end{split}
\end{equation}
the convergence is uniform in $t$ (here $c'$ and $c''$ are positive constants not depending on $t$ or $\delta$).

Adding \eqref{eq:levymeasure-proof-1} and \eqref{eq:levymeasure-proof-2} yields the result.
\end{proof}

Similarly, the generator of a Feller process can be described by a family of L\'evy triplets, i.e., a state space dependent triplet. In this case the analog to $\xi\mapsto\psi(\xi)$ in \eqref{eq:lkf} is a function $(x,\xi)\mapsto q(x,\xi)$ called symbol of the Feller process. But note that (most of) the theory based on symbols requires the test functions $C_c^\infty$ to be in the domain of the generator, cf. \cite{BoetSchiWang2013}. For the processes treated in the current paper, this is not the case. Therefore we had to take a different approach based on the above idea.

\section*{Acknowledgement} While working on coupling of multiplicative L\'evy processes with Anita Behme we became aware of the preprint \cite{Majk2015}. This was the starting point for the current paper. Comments of Ren\'e Schilling helped to improve the presentation. We are also grateful for comments by Wilfrid Kendall and Franziska K\"uhn.

\end{document}